\documentclass[11pt,leqno]{article}
\usepackage{amssymb, amscd, amsmath, amsthm}

\allowdisplaybreaks

\def\ve{\varepsilon}

\def\mod{\,\text{\rm mod}\;}
\def\beq{\begin{equation}}
\def\eeq{\end{equation}}
\def\cite#1{{\rm [#1]}}

\newtheorem{theorem}{Theorem}

\newtheorem{lemma}{Lemma}
\newtheorem{cor}{Corollary}
\theoremstyle{definition}

\newtheorem*{rema}{Remark}

\begin{document}

\numberwithin{equation}{section}
\title{An approximation to the twin prime conjecture and the parity phenomenon}

\author{J\'anos Pintz\thanks{Supported by OTKA Grants K72731, K67676 and ERC-AdG.228005.}}

\date{}
\maketitle
\section{Introduction}
\label{sec:1}

The celebrated theorem of Chen \cite{Che1}, \cite{Che2} proved nearly 50 years ago asserts that there are infinitely many primes $p$ for which $p + 2$ is either a prime or has exactly two prime factors.
In view of this strong approximation of the twin prime conjecture seems to be a surprise that it is not known whether there are infinitely many primes $p$ such that $p + 2$ has an odd number of prime factors.
The reason for it is, as described by Hildebrand (\cite{Hil}) ``the so-called \emph{parity barrier}, a heuristic principle according to which sieve methods cannot differentiate between integers with an even and odd number of prime factors.''
Iwaniec \cite{Iwa} writes similarly about the \emph{parity phenomenon}:
``The parity phenomenon is best explained in the context of Bombieri's asymptotic sieve \cite{Bom}.
This says that within the classical conditions for the sieve one cannot sift but all numbers having the same parity of the number of prime divisors.
Never mind producing primes; we cannot even produce numbers having either one, three, five or seven prime divisors.
However, under the best circumstances we can obtain numbers having either 2006 or 2007 prime divisors.
Similarly we can obtain numbers having either one or two prime divisors, but we are not able to determine which of these numbers are there, probably both.''

In the present work we prove a weaker version of the problem that \hbox{$\lambda(p + 2) = -1$} for infinitely many primes~$p$, where $\lambda(n)$ is Liouville's function:
\beq
\lambda(n) = (-1)^{\Omega(n)}, \quad \Omega(n) = \sum_{p \mid n} 1,
\label{eq:1.1}
\eeq
and as always, $p, p_i, p'$ denote primes, $\mathcal P = \{p_n \}^\infty_{n = 1}$ the set of all primes.

\begin{theorem}
\label{th:1}
There exists an even number $d$ such that $0 < |d| \leq 16$ and $\lambda(p + d) = -1$ for infinitely many primes~$p$.
\end{theorem}

The method proves actually somewhat more.

\medskip
\noindent
{\bf Theorem 1'.}
{\it At least one of the following two assertions is true:

{\rm (i)} There exists an even $d$ with $|d| = 2, 4$ or $8$, such that $\lambda(p + d) = -1$ for infinitely many primes~$p$;

{\rm (ii)} $\liminf\limits_{n\to \infty} (p_{n + 1} - p_n) \leq 16$.}

\begin{rema}
Theorem~1' remains true if $\lambda(p + d)\! =\! -1$ is replaced by $\lambda(p + d)\! =\! 1$.
\end{rema}

An alternative version of Theorem~\ref{th:1} would be

\medskip
\noindent
{\bf Theorem 1''.}
{\it There exists a positive even $d\leq 18$ such that $\lambda(p + d) = -1$ for infinitely many primes $p$.}

\begin{rema}
The condition $2\,|\, d$, $0\! <\! d \!\leq\! 18$ can be replaced here by $0\! <\! d \!\leq\! 17$.
\end{rema}

With a refinement of the original method we can prove Theorem~\ref{th:1} with the additional requirement that $p + d$ should be an almost prime.
Let $P^-(n)$ denote the smallest prime factor of $n$.

\begin{theorem}
\label{th:2}
There exist absolute constants $c$ and $C$, an \emph{odd} $b \leq C$ and a $d$ with $0 < |d| \leq 30$ such that $P^-(p + d) > p^c$, $\Omega(p + d) = b$ for infinitely many primes~$p$.
\end{theorem}

Similarly to Theorem 1' the modified method yields the following result:

\medskip
\noindent
{\bf Theorem 2'.}
{\it At least one of the following two assertions is true:

{\rm (i)} There exist constants $C$, $c$, an \emph{odd} $b \leq C$ and an even $d \neq 0$ such that $|d| \leq 18$ and $P^-(p + d) > p^c$, $\Omega(p + d) = b$ for infinitely many primes~$p$;

{\rm (ii)} $\liminf\limits_{n\to \infty} (p_{n + 1} - p_n) \leq 30$.}

\medskip
All the above results are based on the method \cite{GPY} yielding
\beq
\liminf_{n \to \infty} (p_{n + 1} - p_n) / \log p_n = 0,
\label{eq:1.2}
\eeq
and similarly to the above work we will investigate \emph{admissible} $k$-element sets
\beq
\mathcal H = \{h_i\}^k_{i = 1}, \quad 0 \leq h_1 < h_2 < \dots < h_k, \quad h_i \in \mathbb Z,
\label{eq:1.3}
\eeq
where we call a set $\mathcal H$ admissible if it does not occupy all residue classes modulo any prime.
In fact all the above mentioned results are simple consequences of more general ones referring to admissible sets.

In such a way, Theorems~\ref{th:1}--1'' are corollaries of the following result $(n + \mathcal H = \{n + h_i\}^k_{i = 1})$.

\begin{theorem}
\label{th:3}
Let $\mathcal H$ be an admissible $6$-tuple, $r \neq 0$, $r \notin \mathcal H$ any fixed integer.
Then we have infinitely many integers $n$ such that either

{\rm (i)} $n + \mathcal H$ contains at least two primes, or

{\rm (ii)} $n + h_i$ is prime for some $h_i \in \mathcal H$ and $\lambda(n + r) = -1$.
\end{theorem}

It is easy to see that the set $\mathcal H_0 = \{0, 4, 6, 10, 12, 16\}$ is admissible.
Choosing $r = 8$ we obtain Theorems~\ref{th:1} and 1', while the value $r = 18$ yields Theorem~1'', $r = 17$ the result of the remark following Theorem~1''.

Similarly to the above, Theorem 2' follows from

\begin{theorem}
\label{th:4}
Let $\mathcal H = \{h_i\}^9_{i = 1}$ be an admissible $9$-tuple, $\mathcal H' = \mathcal H \setminus \{h_j\}$ with some fixed $j \in [1, 9]$.
There exist absolute constants $C, c$, an \emph{odd} integer $b \leq C$, such that we have infinitely many integers $n$ with $P^-(n + h_i) > n^c$, $\Omega(n + h_i) \leq C$ for $1 \leq i \leq 9$, $\Omega(n + h_j) = b$ and a $\nu \in [1,9]$, $\nu \neq j$ with $n + h_\nu \in \mathcal P$.
\end{theorem}

It is easy to see that $\mathcal H = \mathcal H_1 = \{0, 2, 6, 8, 12, 18, 20, 26, 30\}$ is an admissible $9$-tuple.
This immediately gives Theorem~\ref{th:2}, while the choice $h_j = 12$ (or $h_j = 18$) yields Theorem~2'.

\section{Conditional theorems}
\label{sec:2}
\setcounter{equation}{0}

A crucial ingredient of the proof is the celebrated Bombieri--Vinogradov theorem, similarly to the proof of \eqref{eq:1.2}.
The number $\vartheta$ is called an admissible level of distribution of primes if for any $\ve > 0$, $A > 0$
\beq
\sum_{q \leq N^{\vartheta - \ve}} \max_{\substack{a\\ (a,q) = 1}} \biggl| \sum_{\substack{p \equiv a \pmod q\\
p \leq N}} \log p - \frac{N}{\varphi(q)} \biggr| \ll_{\ve, A} \frac{N}{\log^A N}.
\label{eq:2.1}
\eeq
The Bombieri--Vinogradov theorem asserts that $\vartheta = 1/2$ is an admissible level, while the Elliott--Halberstam conjecture states that $\vartheta = 1$ is admissible too.
If $\vartheta$ is larger we can get closer to the original conjecture stating $\lambda(p + 2) = -1$ infinitely often.
For $\vartheta > 0.729$ we can prove (see Theorem~\ref{th:7}) the existence of infinitely many pairs $n_1, n_2$ with $|n_1 - n_2| \leq 2$, $n_1 \in \mathcal P$, $\lambda(n_2) = -1$.
However, even assuming the Elliott--Halberstam conjecture we cannot prove the existence of infinitely many primes $p$ with $\lambda(p + d) = -1$ for even a single a priori given~$d$.

The possible conditional improvements over Theorems~\ref{th:1}--\ref{th:4} depend on our knowledge of $\vartheta$, the level of distribution of primes.
However, we need also an assumption about the $\lambda$-function, analogous to \eqref{eq:2.1},
namely we suppose the existence of a $\vartheta$ for which besides \eqref{eq:2.1}
also
\beq
\sum_{q \leq N^{\vartheta - \ve}} \max_{y \leq N} \max_a \biggl| \sum_{\substack{n \equiv a\pmod q\\
n \leq y}} \lambda(n) \biggr| \ll_{\ve, A} \frac{N}{(\log N)^A}
\label{eq:2.2}
\eeq
holds.
We will show analogues of Theorems \ref{th:3}--\ref{th:4} for the conditional case $\vartheta > 1/2$.

\begin{theorem}
\label{th:5}
Let $\mathcal H$ be an admissible $k$-tuple, $k = C_1(\vartheta)$, $r \neq 0$, $r \notin \mathcal H$ any fixed integer.
Then we have infinitely many integers $n$ such that either

{\rm (i)} $n + \mathcal H$ contains at least two primes, or

{\rm (ii)} $n + h_i$ is prime for some $h_i \in \mathcal H$ and $\lambda(n + r) = -1$.

\noindent
The above holds with $C_1(0.729) = 2$, $C_1(0.616) = 3$, $C_1(0.554) = 4$ and $C_1(0.515) = 5$.
\end{theorem}

\begin{theorem}
\label{th:6}
Let $\mathcal H = \{h_i\}^k_{i = 1}$ be an admissible $k$-tuple, $k = C_2(\vartheta)$, $\mathcal H' = \mathcal H \setminus \{h_j\}$ for some fixed $j \in [1, C_2(\vartheta)]$.
There exist absolute constants $C$, $c$, an \emph{odd} $b \leq C$ such that we have infinitely many integers $n$ with $P^-(n + h_i) > n^c$, $\Omega(n + h_i) \leq C$ for $1 \leq i \leq C_2(\vartheta)$, $\Omega(n + h_j) = b$ and a $\nu \in [1, C_2(\vartheta)]$ such that $\nu \neq j$, $n + h_\nu \in \mathcal P$.
We can choose here $C_2(0.924) = 3$, $C_2(0.739) = 4$, $C_2(0.643) = 5$, $C_2(0.584) = 6$, $C_2(0.544) = 7$, $C_2(0.516) = 8$.
\end{theorem}

The consequences of Theorem~\ref{th:5} (analogously to Theorems~\ref{th:1}--1'') are the following.

\begin{theorem}
\label{th:7}
There exists an integer $d$, such that $0 < |d| \leq C_3(\vartheta)$ and $\lambda(p + d) = - 1$ for infinitely many primes~$p$.
We can choose here $C_3(0.729) = 2$, $C_3(0.616) = 6$, $C_3(0.554) = 8$, $C_3(0.515) = 12$.
\end{theorem}

\begin{rema}
Apart from the first case $\vartheta = 0.729$ we can assume $r$ to be even.
\end{rema}

\noindent
{\bf Theorem 7'.}
{\it At least one of the following two assertions is true:

{\rm (i)} There exists a $|d| \leq C_4(\vartheta)$ such that $\lambda(p + d) = -1$ for infinitely many primes~$p$, where $C_4(0.729) = 1$, $C_4(0.616) = 3$, $C_4(0.554) = 4$, $C_4(0.515) = 7$.

{\rm(ii)} $\liminf\limits_{n \to \infty} (p_{n + 1} - p_n) \leq C_3(\vartheta)$, with $C_3(\vartheta)$ as in Theorem~\ref{th:7}.}

\begin{rema}
Theorem~7' remains true if $\lambda(p + d) = -1$ is replaced by $\lambda(p + d) = 1$ in (i).
\end{rema}

\noindent
{\bf Theorem 7''.}
{\it There exists a positive even $d \leq C_3(\vartheta) + 2$ such that $\lambda(p + d) = -1$ infinitely often.}

\begin{rema}
The condition $2\mid d$, $0 < d \leq C_3(\vartheta) + 2$ can be replaced here by $0 < d \leq C_3(\vartheta) + 1$.
\end{rema}

The corollaries of Theorem~\ref{th:6} (analogously to Theorems~\ref{th:2}--2') are the following.

\begin{theorem}
\label{th:8}
There exist absolute constants $C$, $c$, an \emph{odd} $b \leq C$ and an even number $d$ such that $0 < |d| \leq C_5(\vartheta)$ and $P^-(p + d) > p^c$, $\Omega(p + d) = b$ for infinitely many primes~$p$.
We can choose here $C_5(0.924) = 6$, $C_5(0.739) = 8$, $C_5(0.643) = 12$, $C_5(0.584) = 16$, $C_5(0.544) = 20$, $C_5(0.516) = 26$.
\end{theorem}

\noindent
{\bf Theorem 8'.}
{\it At least one of the following assertions are true:

{\rm (i)} There exist absolute constants $C$, $c$, an \emph{odd} $b\leq C$ and an even $d$ such that $0 < |d| \leq C_6(\vartheta)$ and
$P^-(p + d) > p^c$,
$\Omega(p + d) = b$ for infinitely many primes~$p$.
We can choose here $C_6(0.924) = 4$, $C_6(0.739) = 6$, $C_6(0.643) = 6$,
$C_6(0.584) = 10$,
$C_6(0.544) = 12$, $C_6(0.516) = 14$.

{\rm (ii)} $\liminf (p_{n + 1} - p_n) \leq C_5(\vartheta)$ with the $C_5(\vartheta)$ given in Theorem~\ref{th:8}.}

\begin{rema}
We emphasize here that under the strongest assumptions on $\vartheta$ we obtained the following assertions:

A) If $\vartheta > 0.729$, then we have a $d$ with $0 < |d| \leq 2$ such that $\lambda(p + d) = -1$ for infinitely many primes~$p$.

B) If $\vartheta > 0.729$, then either the twin prime conjecture is true or $\lambda(p + d) = -1$, $|d| = 1$ holds for infinitely many primes~$p$.

C) If $\vartheta > 0.924$, then there exist absolute constants $C$, $c$, an \emph{odd} $b < C$ and a $d$ with $0 < |d| \leq 6$ such that $P^-(p + d) > p^c$, $\Omega(p + d) = b$ for infinitely many primes~$p$.

D) If $\vartheta > 0.924$, then either $\liminf\limits_{n \to \infty} (p_{n + 1} - p_n) \leq 6$ or there exists a $C$ and an \emph{odd} $b < C$ such that $\Omega(p + d) = b$ holds with $d = 2$ or $d = -4$ for infinitely many primes $p$ (or alternatively one can choose $d = 4$ or $d = -2$).
\end{rema}

In order to see that Theorems~\ref{th:5} and \ref{th:6} imply the later results we have only to note that for $2 \leq k \leq 8$ we have the following admissible sets:
\begin{align*}
\mathcal H_2 &= \{0, 2\},\\
\mathcal H_3 &= \{0, 2, 6\} \ \text{ or }\ \{0, 4, 6\},\\
\mathcal H_4 &= \{0, 2, 6, 8\},\\
\mathcal H_5 &= \{0, 4, 6, 10, 12\},\\
\mathcal H_6 &= \{0, 4, 6, 10, 12, 16\},\\
\mathcal H_7 &= \{0, 2, 6, 8, 12, 18, 20\},\\
\mathcal H_8 &= \{0, 2, 6, 8, 12, 18, 20, 26\}.
\end{align*}

\begin{rema}
One can show that for $k \leq 8$ these are the sets with minimal diameter, that is with minimal value of $h_k - h_1$ in \eqref{eq:1.3}.
\end{rema}

In order to conclude Theorem~\ref{th:7} from Theorem~\ref{th:5} we can choose any $r \notin \mathcal H_k$ with $0 < r < h_k$.
For Theorem~7' we choose $r$ so that $|r - h_k / 2|$ should be minimal, for Theorem~7'' $r = h_k + 2$ (or $h_k + 1$ for the result in the remark after Theorem~7'').
On the other hand, when deducing Theorem~\ref{th:8} from Theorem~\ref{th:6} we choose $r = h_j$ as any element of~$\mathcal H_k$.
To obtain Theorem~8' we choose it so that $|r - h_k/2|$ should be as small as possible.

Finally, if we would like to approximate the generalized twin prime problem ($p, p + 2d$ are both primes infinitely often for any integer $d > 0$), then we might consider the following two admissible sets for any $m \in \mathbb Z$:
\beq
\mathcal H'_2 = \{0, 2d\}, \quad \mathcal H''_3 = \{0, 6m, 12m\},
\label{eq:2.3}
\eeq
which yield the following corollaries to Theorems~\ref{th:5} and \ref{th:6}.

\begin{cor}
\label{cor:1}
Assume $\vartheta > 0.729$.
For any non-zero even $d$ we have either

{\rm (i)} infinitely many prime pairs $\{p, p + 2d\}$ or

{\rm (ii)} infinitely many pairs $n_1, n_2 \in \mathbb Z$ with $n_1$ being prime, $\lambda(n_2) = -1$, $|n_1 - n_2| = d$.
\end{cor}

\begin{proof}
Choose $r = d$ in Theorem~\ref{th:5}.
\end{proof}

\begin{cor}
\label{cor:2}
Assume $\vartheta > 0.924$.
For any $m \neq 0$ we have absolute constants $C$, $c$, an \emph{odd} $b \leq C$ such that we have either

{\rm (i)} infinitely many prime pairs $\{p, p + 12 m\}$ or

{\rm (ii)} infinitely many pairs $n_1, n_2 \in \mathbb Z$ with $n_1$ being prime, $P^-(n_2) > n^c_2$, $\Omega(n_2) = b$, $|n_1 - n_2| = 6m$.
\end{cor}

\begin{proof}
Choose $r = h_2 = 6m$ in Theorem~\ref{th:6}.
\end{proof}

\section{Proofs of Theorems~\ref{th:3} and \ref{th:5}}
\label{sec:3}

The idea of the proof is -- analogously to \cite{GPY} -- to weigh the natural numbers with a weight inspired by Selberg's sieve (cf.\ (2.13) of \cite{GPY})
\beq
\Lambda_R(n; \mathcal H, l) := \frac1{(k + l)!} \sum_{\substack{d \mid P_{\mathcal H}(n)\\
d \leq R}} \mu(d) \left(\log \frac{R}{d}\right)^{k + l}, \quad P_{\mathcal H}(n) = \prod^k_{i = 1} (n + h_i),
\label{eq:3.1}
\eeq
more precisely by the square of a linear combination of these weights
\beq
a_n := a_n(\mathcal H; u) := \left(\Lambda_R(n; \mathcal H, 0) + \frac{u(k + 1)}{\log R} \Lambda_R(n; \mathcal H, 1)\right)^2,
\label{eq:3.2}
\eeq
where $u$ is a real parameter to be optimized according to the concrete problem.
We will choose $N$ as a large integer tending to infinity and let $n$ run in the interval $[N, 2N)$ which we abbreviate by $n \sim N$.
We will put
\beq
\chi_{\mathcal P}(n) := \begin{cases}
1 &\text{if }\ n \in \mathcal P\\
0 &\text{if }\ n \notin \mathcal P
\end{cases}, \
\chi_\lambda(n) := \frac{1 - \lambda(n)}{2} = \begin{cases}
1 &\text{if }\ \lambda(n) = -1\\
0 &\text{if }\ \lambda(n) = 1
\end{cases},
\label{eq:3.3}
\eeq
and study the average of the function $a_n s(n)$, namely,
\beq
S(N, \mathcal H, u) := \frac1{N}\! \sum_{n \sim N} a_n s(n), \ 
s(n)\! :=\! s_{\mathcal P}(n) + \chi_\lambda(n + r), \ 
s_{\mathcal P}(n) \! :=\! \sum^k_{i = 1} \chi_{\mathcal P}(n + h_i).
\label{eq:3.4}
\eeq

We will compare this quantity with the average of the weights $a_n$, that is, with
\beq
A(N, \mathcal H, u) := \frac1{N} \sum_{n \sim N} a_n.
\label{eq:3.5}
\eeq

Our goal will be to show
\beq
S(N, \mathcal H, u) > A(N, \mathcal H, u)
\label{eq:3.6}
\eeq
which clearly implies the existence of at least one $n \sim N$ with $s(n) > 1$ and thereby the existence of either

(i) two primes of the form $n + h_i$, $n + h_j$ $(i \neq j)$ or

(ii) one prime $n + h_i$ and $\lambda(n + r) = -1$ $(r \neq h_i)$.

In the proof of \eqref{eq:3.6} we can make use of Propositions 1 and 2 of \cite{GPY}, which we quote now as Lemmas~\ref{lem:1} and \ref{lem:2} in the special case $\mathcal H_1 = \mathcal H_2$, $h_k \leq C(k) = O(1)$ as $k$ will be bounded in our case.
Constants $c, C, c_i, C_i$ will be absolute unless otherwise stated and can be different at different occurrences.
The same is true for constants implied by the $\ll$ or $O$ symbols.
The symbol $o$ refers to the case $N \to \infty$, but it might also depend on~$k$.
The letters $p$, $p_i$ will denote always primes.

The crucial singular series is defined for $\mathcal H = \mathcal H_k$ by
\beq
\mathfrak S(\mathcal H) := \prod_p \left(1 - \frac{\nu_p}{p}\right) \left(1 - \frac1p\right)^{-k} \geq c_0(k, \mathcal H),
\label{eq:3.7}
\eeq
for admissible $\mathcal H$
where $\nu_p = \nu_p(\mathcal H)$ denotes the number of residue classes occupied by $\mathcal H \mod p$.
As stated in Theorems \ref{th:3}--\ref{th:6}, we will always assume that $\mathcal H$ is admissible, that is,
\beq
\nu_p < p\ \text{ for }\ p \in \mathcal P, \ \text{ equivalently }\ \mathfrak S (\mathcal H) \neq 0.
\label{eq:3.8}
\eeq

With these notations we state (cf.\ (2.14)--(2.15) of \cite{GPY})

\begin{lemma}
\label{lem:1}
If $R \ll N^{1/2} (\log N)^{-8(k + 1)}$ then
\beq
\label{eq:3.9}
\frac1N \sum_{n \sim N} \Lambda_R(n; \mathcal H, l_1) \Lambda_R(n; \mathcal H, l_2) =
\bigl(\mathfrak S(\mathcal H) + o(1)\bigr) {l_1 + l_2\choose l_1} \frac{(\log R)^{k + l_1 + l_2}}{(k + l_1 + l_2)!}.
\eeq
\end{lemma}

\begin{lemma}
\label{lem:2}
If $A > 0$ arbitrary, $R \ll N^{\vartheta /2} (\log N)^{-C(A, k)}$ then
for any $h \in \mathcal H$ we have
\begin{align}
\label{eq:3.10}
&\frac1N \sum_{n \sim N} \Lambda_R(n; \mathcal H, l_1) \Lambda_R(n; \mathcal H, l_2) \chi_{\mathcal P} (n + h) =\\
&=\bigl(\mathfrak S(\mathcal H) + o(1)\bigr) {l_1 + l_2 + 2\choose l_1 + 1} \frac{(\log R)^{k + l_1 + l_2 + 1}(\log N)^{-1}}{(k + l_1 + l_2 + 1)!}.
\notag
\end{align}
\end{lemma}

These lemmas take care of the evaluation of the averages of $a_n$ and $a_n \chi_{\mathcal P}(n + h)$, so we have to deal only with the average of $a_n \lambda(n + r)$.
This is in principle similar to the case $a_n \chi_{\mathcal P}(n + h)$ but in fact it is much easier.
Although its evaluation reflects a deep property of the $\lambda$-function, it is (unlike the case of primes) a simple consequence of the analogue of the Bombieri--Vinogradov theorem, that is, of \eqref{eq:2.2}.
In case $\vartheta > 1/2$ this is an unproved condition, which we assume in Theorems~\ref{th:5}--\ref{th:8}, while for $\vartheta = 1/2$ we can state it as

\begin{lemma}
\label{lem:3}
Relation \eqref{eq:2.2} is true for $\vartheta = 1/2$.
\end{lemma}

Its proof runs completely analogously to Theorem~4 of Vaughan \cite{Vau} which is the analogous assertion for the M\"obius function $\mu(n)$ in place of Liouville's function $\lambda(n)$.
This implies easily the ``analogue'' of Lemma~\ref{lem:2} for the $\lambda$-function.

\begin{lemma}
\label{lem:4}
If $A > 0$ arbitrary, $R \ll N^{\vartheta/2}(\log N)^{-C(A, k)}$ then for any $r \leq N$ we have
\beq
\frac1N \sum_{n \sim N} \Lambda_R(n; \mathcal H, l_1) \Lambda_R(n; \mathcal H, l_2) \lambda(n + r) \ll_{k, l, A} (\log N)^{-A} .
\label{eq:3.11}
\eeq
\end{lemma}

\begin{proof}
For any squarefree $m$ we have for the number $\nu_m = \nu_m(\mathcal H)$ of solution of
\beq
P_{\mathcal H}(n) \equiv 0 \pmod m
\label{eq:3.12}
\eeq
by the Chinese remainder theorem
\beq
\nu_m = \prod_{p \mid m} \nu_p \leq k^{\omega(m)} = d_k(m),
\label{eq:3.13}
\eeq
where $\omega(m)$ denotes the number of different prime divisors of~$m$, $d_k(m)$ the number of ways to write $m$ as a product of $k$ integers.
Therefore, interchanging the summation in \eqref{eq:3.11} we obtain from \eqref{eq:2.2} for the left-hand side of \eqref{eq:3.11}
\begin{align}
\label{eq:3.14}
&\frac1N \sum_{d \leq R} \sum_{e \leq R} \frac{\mu(d)\mu(e) \left(\log \frac{R}{d} \right)^{k + l_1} \left(\log \frac{R}{e}\right)^{k + l_2}} {(k + l_1)!(k + l_2)!} \sum_{\substack{n \sim N\\ [d, e] \mid P_{\mathcal H}(n)}} \lambda(n + r) \ll \\
&\ll \frac{(\log R)^{2k + 2}}{N} \sum_{q \leq R^2} \biggl(\sum_{q = [d, e]} 1 \biggr) \nu_q \, E_{3N}(q) \notag
\end{align}
where (for $q \leq M$)
\beq
E_M(q) := \max_{y \leq M} \max_a \biggl|\sum_{\substack{ n\equiv a \pmod q\\
n \leq y}} \lambda(n) \biggr| \ll \frac{M}{q}.
\label{eq:3.15}
\eeq

Now, by \eqref{eq:2.2} and \eqref{eq:3.13}, the sum over $q$ in \eqref{eq:3.14} is, similarly to (9.13) of \cite{GPY},
\begin{align}
\label{eq:3.16}
&\ll \biggl( \sum_{q \leq R^2} \frac{d_{3k}(q)^2}{q} \sum_{q \leq R^2} q_{3N} E^2_{3N} (q)\biggr)^{1/2}\ll\\
&\ll \left((\log N)^{9k^2} N \cdot \frac{N}{\log^A\!\! N} \right)^{1/2} \ll N(\log N)^{(9k^2 - A)/2}.
\notag
\end{align}
\end{proof}

Using the notation
\beq
B := B(R, \mathcal H, k) := \frac{\mathfrak S(\mathcal H)\log^k R}{k!}
\label{eq:3.17}
\eeq
we have by Lemma~\ref{lem:1}
\beq
A(N, \mathcal H, u) \sim B \left(1 + 2u + 2u^2 \cdot \frac{k + 1}{k + 2} \right).
\label{eq:3.18}
\eeq

On the other hand, Lemma~\ref{lem:2} implies if $R = N^{(\vartheta - \ve)/2}$
\begin{align}
\label{eq:3.19}
S_{\mathcal P}(N, \mathcal H, u) :&= \frac1N \sum_{n \sim N} a_n s_{\mathcal P}(n) \sim \\
&\sim \frac{B k(\vartheta - \ve)}{2} \left( \frac2{k + 1} + \frac{6u}{k + 2} + \frac{6u^2(k + 1)}{(k + 2)(k + 3)} \right).\notag
\end{align}

Finally, Lemma~\ref{lem:4} yields
\beq
S_\lambda(N, \mathcal H, u) := \frac1N \sum a_n \chi_\lambda(n + r) \sim \frac{A(N, \mathcal H, u)}{2}.
\label{eq:3.20}
\eeq

It follows from \eqref{eq:3.3}--\eqref{eq:3.4} and \eqref{eq:3.17}--\eqref{eq:3.20} that in order to show the crucial relation $S(N, \mathcal H, u) > A(N, \mathcal H, u)$ we have to find a value $u$ such that
\beq
S_{\mathcal P}(N, \mathcal H, u) > \frac{1 + \ve}{2} A(N, \mathcal H, u)
\label{eq:3.21}
\eeq
which is satisfied if
\beq
k\vartheta \left( \frac2{k + 1} + \frac{6u}{k + 2} + \frac{6 u^2 (k + 1)}{(k + 2) (k + 3)} \right) > 1 + 2u + \frac{2u^2(k + 1)}{k + 2}.
\label{eq:3.22}
\eeq

If we want to show Theorem~\ref{th:3} (in which case $\vartheta = 1/2$) with $k = 6$ we have to prove the existence of a $u$ with
\beq
\frac67 + \frac{9u}{4} + \frac{7u^2}{4} - \left(1 + 2u + \frac{7u^2}{4}\right) > 0
\label{eq:3.23}
\eeq
which is equivalent with $u > 4/7$.
This proves Theorem~\ref{th:3}.

\bigskip
In order to prove Theorem~\ref{th:5} we have to consider for $k = 2, 3, 4, 5$ a quadratic inequality for $u$ and calculate the value $\vartheta_0$ for which the discriminant of the (in general really quadratic) formula equals~$0$.
If $\vartheta_0 \geq 1$ the parameter~$k$, the size of our set $\mathcal H$ is too small.
If $\vartheta_0 < 1/2$, we have an unconditional solution.
If $1/2 \leq \vartheta_0 < 1$, Theorem~\ref{th:5} is true for $\vartheta > \vartheta_0$.

\begin{rema}
Alternatively we can calculate the maximum of the ratio of the left- and right-hand side (taken without $\vartheta$) of \eqref{eq:3.22} and choosing the optimal value of $u$ we get a lower bound for~$\vartheta$.
\end{rema}

\begin{rema}
It is easy to see from \eqref{eq:3.21}--\eqref{eq:3.22} that working with the pure $k$-dimensional sieve corresponding to $l = 0$ in the weight function \eqref{eq:3.1} (without using any other values $l$ which corresponds to taking $u = 0$ in \eqref{eq:3.2}) we are not able to prove any unconditional result, even for arbitrarily large value of~$k$.
\end{rema}

\section{Proofs of Theorems~\ref{th:4} and \ref{th:6}}
\label{sec:4}

In case of the proofs of Theorems \ref{th:4} and \ref{th:6} we will choose one specific element $h_j \in \mathcal H$ and try to produce almost primes (with $\Omega(n + h_i) \leq C$) in all components $\{n + h_i\}^k_{i = 1}$, and additionally either

(i) at least two primes $n + h_\nu$, $n + h_\mu$ with $\nu, \mu \in [1,k] \setminus \{j\}$, or

(ii) one prime $n + h_\nu$ with $\nu \in [1,k]$, $\nu \neq j$ and $\lambda(n + h_j) = -1$.

The starting point is to produce almost primes in each components, which was shown to be possible in \cite{Pin} in such a way that
 using the weights \eqref{eq:3.1}--\eqref{eq:3.2},
 the total measure of those numbers $n$ for which $\mathcal P_{\mathcal H}(n) = \prod\limits^k_{i = 1}(n + h_i)$ had a prime factor below $N^\eta$ was negligible compared with the total measure of all $n \sim N$ if $k$ and $\mathcal H_k$ were fixed, $N
\to \infty$ and $\eta \to 0$.
Denoting $P(m) = \prod\limits_{p \leq m} p$ this was formulated (cf.\ Lemma~4 of \cite{Pin}) in the following way.

\begin{lemma}
\label{lem:5}
Let $N^{C_0} < R \leq N^{1/(2 + \eta)}(\log N)^{-C}$, $\eta > 0$.
Then
\beq
\sum_{\substack{n \sim N\\
(P_{\mathcal H}(n), P(R^\eta)) > 1}} \Lambda_R(n; \mathcal H, l)^2 \ll \eta \sum_{n \sim N} \Lambda_R(n; \mathcal H, l)^2,
\label{eq:4.1}
\eeq
where the constants $C$ and the one implied by the $\ll$ symbol may depend on $k$, $l$ and~$\mathcal H$.
\end{lemma}

As remarked in \cite{Pin} after the formulation of Lemma~\ref{lem:4}, the analogous quantity with the product of $\Lambda_R(n; \mathcal H, l_1)$ and $\Lambda_R(n; \mathcal H, l_2)$ with different values of $l_1$ and $l_2$ can be estimated by Cauchy's inequality and so we arrive at the more general relation
\beq
\frac1N \sum_{\substack{n \sim N\\
(P_{\mathcal H}(n), P(R^\eta)) > 1}} a_n(\mathcal H, u) \ll \frac{\eta}{N} \sum_{n \sim N} a_n(\mathcal H, u),
\label{eq:4.2}
\eeq
or equivalently (cf.\ \eqref{eq:3.2})
\beq
\frac1N\!\! \sum_{\substack{n \sim N\\
(P_{\mathcal H}(n), P(R^\eta)) = 1}}\!\! \!\!\! a_n(\mathcal H, u) = \left(\frac{1 \!  +\!  O(\eta)}{N}\right) \sum_{n \sim N} a_n(\mathcal H, u) =: (1 + O(\eta)) A(N, \mathcal H, u),
\label{eq:4.3}
\eeq
where the constants implied by the $\ll$ and $O$ symbols may now depend on $u$, too.
Since $s(n) \leq k \leq 9$, \eqref{eq:4.2} immediately implies
\beq
S^*(N, \mathcal H, u) = \frac1N \sum_{\substack{n \sim N\\
(P_{\mathcal H}(n), P(R^\eta)) = 1}} a_n s(n) = S(N, \mathcal H, u) + O(\eta A(N, \mathcal H, u)).
\label{eq:4.4}
\eeq
This means that apart from a factor $1 + O(\eta)$ the situation is the same as in the previous section, the only difference being that we have just $k - 1$ components with primes instead of~$k$.
This means that choosing $\eta$ as a sufficiently small constant in place of \eqref{eq:3.22} we have now to ensure the inequality
\beq
(k - 1)\vartheta \left( \frac2{k + 1} + \frac{6u}{k + 2} + \frac{6u^2(k + 1)}{(k + 2)(k + 3)} \right) > 1 + 2u + \frac{2u^2(k + 1)}{(k + 2)}.
\label{eq:4.5}
\eeq

In case of Theorem~\ref{th:4} we have $\vartheta = 1/2$ and for $k = 9$ we have to find a $u$ with
\beq
\frac45  + \frac{24u}{11} + \frac{20u^2}{11} - \left( 1 + 2u + \frac{20u^2}{11}\right) > 0,
\label{eq:4.6}
\eeq
which is equivalent with $u > 1.1$.
This proves Theorem~\ref{th:4}.

Theorem~\ref{th:6} can be shown similarly to Theorem~\ref{th:5}, as described at the end of the previous section.

\section{Concluding remarks}
\label{sec:5}

The more general formulation of Theorems~\ref{th:3}--\ref{th:4} shows that apart from the small $d$'s in Theorems~\ref{th:1}--\ref{th:2} we obtain actually many different even values of $d$ such that $\lambda(p + d) = 1$ for infinitely many primes~$p$.
In fact, the lower density $\bold d(\mathcal D_0)$ of the corresponding set $\mathcal D_0$ of such $d$'s is positive and it is easy to give an explicit lower bound for it as well.
The argument is simple and similar to the one of Section~11 in \cite{Pin}.
In general, suppose $k$ is bounded and $\mathcal H_k \subset [1, U]$, where we may assume $P := P(k) := \prod\limits_{p \leq k} p \mid U$, since $U \to \infty$.
If we choose all elements from the set
\beq
\mathcal M := \{m \leq U; \ (m, P) = 1\}, \ \text{ where }\ M := |\mathcal M| = \frac{\varphi(P)}{p} U,
\label{eq:5.1}
\eeq
then $\mathcal H_k$ will be admissible since the zero residue class is not covered by $\mathcal H$ modulo any $p \leq k$.
Taking all choices for $\mathcal H_k$ we obtain ${M\choose k}$ even values of~$d$, counted with multiplicity.
A fixed difference $d$ implies at most $M - 1$ choices for the pair $h_i, h_j$ with $h_i - h_j = d$ and afterwards we have ${M - 2\choose k - 2}$ choices for the remaining $k - 2$ elements.
This implies for the multiplicity of any~$d$ the upper bound $(M - 1) {M - 2\choose k - 2}$, hence we obtain at least
\beq
\frac{M}{k(k - 1)} = \frac{\varphi(P)}{P} \cdot \frac{U}{k(k - 1)}
\label{eq:5.2}
\eeq
suitable even $d$'s, hence with the notation (let $c = c(k)$ be a small fixed constant)
\begin{align}
\label{eq:5.3}
\mathcal D_0 &= \bigl\{ 2\mid d, \# \{p; \lambda(p + d) = -1\} = \infty\bigr\}\\
\mathcal D_1 &= \bigl\{2 \mid d, \#\{p; \lambda(p + d) = -1, \ P^-(p + d) > p^c\} = \infty\bigr\}\notag
\end{align}
in the unconditional case $k = 9$ this yields with the choice $j = k$ in Theorem~\ref{th:4}
\beq
\bold d(\mathcal D_1) \geq \frac1{315},
\label{eq:5.4}
\eeq
so that more than $0.6\%$ of the even numbers have this property.

If we consider the conditional cases too, then the case $\vartheta > 0.927$ can be treated in a more efficient way, if we take into account that for any integers $s$ and $t$ with $0 < r < 3t$ the system $\mathcal H_3 = \{0, 2s, 6t\}$ is admissible, since $\nu_2(\mathcal H) = 1$, $\nu_3(\mathcal H) \leq 2$, $\nu_p(\mathcal H) \leq 3$ for $p > 3$.
Let us choose $h_j = 6t$ in Theorem~\ref{th:6}.
Then at least one of the three numbers $2s$, $6t - 2s$, $6t \in \mathcal D_1$.
Let $U$ be a large number and let
\beq
t^* = \max\{t;\ 6t \leq U, \ 6t \notin \mathcal D_1\}.
\label{eq:5.5}
\eeq
Then clearly for every $s$ at least one of $2s$ and  $6t^* - 2s$ belongs to $\mathcal D_1$; further $6t \in \mathcal D_2$ if $t^* < t \leq U/6$.
This yields in total at least
\beq
\lfloor U/6\rfloor  - t^* + \left\lfloor \frac{6t^*}{4}\right\rfloor \geq \lfloor U/6\rfloor
\label{eq:5.6}
\eeq
elements, hence
\beq
\bold d (\mathcal D_1) \geq 1/6 \ \text{ if } \ \vartheta > 0.927,
\label{eq:5.7}
\eeq
that is, essentially at least one third of all even integers belong to~$\mathcal D_1$.

For $4 \leq k \leq 8$ choosing again $j = k$ in Theorem~\ref{th:6} we obtain from \eqref{eq:5.2}
\begin{equation}
\label{eq:5.8}
\alignedat2
\bold d(\mathcal D_1) &\geq 1/36 \ & &\text{ for } \ \vartheta > 0.739,\\
\bold d(\mathcal D_1) &\geq 1/75 \ & &\text{ for } \ \vartheta > 0.643,\\
\bold d(\mathcal D_1) &\geq 2/225 \ & &\text{ for } \ \vartheta > 0.584,\\
\bold d(\mathcal D_1) &\geq 4/735 \ & &\text{ for } \ \vartheta > 0.547,\\
\bold d(\mathcal D_1) &\geq 1/245 \ & &\text{ for } \ \vartheta > 0.516.
\endalignedat
\end{equation}

Concerning the lower density of $\mathcal D_0$ in the case of Theorem~\ref{th:3}, we start with an admissible $\mathcal H_k \subseteq [1, U]$ and with any even $r \in [1, U] \setminus \mathcal H_k$.
This system yields at least one $d \in \mathcal D_0$ with either
\[
\alignedat3
&\text{(i)} \ & d &= h_j - h_i, \quad & &1 \leq i < j \leq k \ \text{ or}\hspace*{56.35mm}\\
&\text{(ii)} \ & d &= |r - h_i|, \quad & &1 \leq i \leq k.
\endalignedat
\]
We again start with the assumption $P(k) | U$, use notation \eqref{eq:5.1}, choose $\mathcal H_k \subset \mathcal M$ and obtain
\beq
{M\choose k} \left(\left\lfloor \frac{U}{2}\right\rfloor - k \right) \sim \frac{U}{2} {M \choose k} := Y
\label{eq:5.9}
\eeq
such pairs $(\mathcal H_k, r)$.

Any given $d$ might arise in the way (i) at most $\frac{U}{2} (M - 1) {M - 2\choose k - 2}$ ways as in \eqref{eq:5.2}.
On the other hand if $d$ arises as in (ii), then we have ${M \choose k}$ choices for $\mathcal H_k$, $k$ choices for $i$ and afterwards two choices for $r$, altogether $2k{M \choose k}$ possibilities.
This gives for the multiplicity of $d$ the upper bound (for both cases (i) and (ii) together)
\beq
\frac{U}{2}(M - 1) {M - 2\choose k - 2} + 2k {M \choose k} = Y \left(\frac{k(k - 1)}{M} + \frac{4k}{U} \right) := Z.
\label{eq:5.10}
\eeq
Hence the total number of different values $d \in \mathcal D_0$ is $\geq Y / Z$, and consequently
\beq
\bold d(\mathcal D_0) \geq \frac{Y}{Z} \cdot \frac1U = \left(\frac{ZU}{Y}\right)^{-1} = \left( 4k + \frac{k(k - 1)P}{\varphi(P)}\right)^{-1}.
\label{eq:5.11}
\eeq
This yields for the unconditional case $k = 6$ the estimate
\beq
\bold d(\mathcal D_0) \geq 2/273,
\label{eq:5.12}
\eeq
which means that more than $1.4\%$ of all even numbers $d$ belong to $\mathcal D_0$, equivalently satisfy $\lambda(p + d) = -1$ for infinitely many primes~$p$.
The conditional cases $k \leq 5$ follow from \eqref{eq:5.11}:
\begin{equation}
\label{eq:5.13}
\alignedat2
\bold d(\mathcal D_0) &\geq 1/12 \ & &\text{ for } \ \vartheta > 0.729,\\
\bold d(\mathcal D_0) &\geq 1/30 \ & &\text{ for } \ \vartheta > 0.616,\\
\bold d(\mathcal D_0) &\geq 1/52 \ & &\text{ for } \ \vartheta > 0.554,\\
\bold d(\mathcal D_0) &\geq 1/95 \ & &\text{ for } \ \vartheta > 0.515.
\endalignedat
\end{equation}
This means that under the strongest condition $\vartheta > 0.729$, for example, one sixth of all even integers $d$ have the property that $\lambda(p + d) = -1$ for infinitely many primes~$p$.
On the other hand Corollary~\ref{cor:1} shows under the same assumption $\vartheta > 0.729$ that for any even $d$ either $d$ or $2d$, consequently at least half of the even integers appear as the difference of two numbers among with at least one is a prime and the other has an odd number of prime factors.

A further interesting problem is that if we already know that $d \in \mathcal D_0$ or $d \in \mathcal D_1$ then what sort of lower bound can be given for the number of primes below $N$ with the property $\lambda(p + d) = -1$.
In case of $\mathcal D_0$ (Theorem~\ref{th:3}) the proof implies only the weak lower bound
\beq
N \exp \left( -C \frac{\log N}{\log\log N} \right).
\label{eq:5.14}
\eeq
We can say much more in case of $\mathcal D_1$ (Theorem~\ref{th:4}, or in the conditional case Theorem~\ref{th:6}).
If we have, namely, for a $k \in [3, 9]$ with a given $n \sim N$ $P^-\bigl(P_{\mathcal H_k}(n)\bigr) > N^c$, then $P_{\mathcal H_k}(n)$ has $O_{k, c}(1)$ divisors, consequently
\beq
a_n(\mathcal H; u) \ll_u (\log R)^{2k} \ll (\log R)^{2k},
\label{eq:5.15}
\eeq
since $u$ was chosen bounded always.

Taking into account that we obtained in \eqref{eq:4.4} finally (cf.\ \eqref{eq:3.14}--\eqref{eq:3.23})
\beq
\frac1N \sum_{\substack{n \sim N\\
(P_{\mathcal H}(n), P(R^\eta)) = 1}} a_n \biggl( \sum^k_{\substack{i = 1\\
i \neq j}} \chi_{\mathcal P}(n + h_i) + \chi_\lambda(n + h_j) - 1 \biggr) \gg_{k,u} B = \frac{\mathfrak S(\mathcal H) \log^k R}{k!},
\label{eq:5.16}
\eeq
this implies by \eqref{eq:5.15} that we found
\beq
\gg_{k,u} \frac{\mathfrak S(\mathcal H)N}{\log^k N}
\label{eq:5.17}
\eeq
integers $n \in [N, 2N)$ with the property
\beq
\sum^{k}_{\substack{i = 1\\
i\neq j}} \chi_{\mathcal P}(n + h_i) + \chi_\lambda(n + h_j) > 1, \quad
P^-(P_{\mathcal H}(n)) > R^\eta \geq N^c
\label{eq:5.18}
\eeq
since $R \geq N^{1/5}$ was chosen.
We can take here unconditionally $k = 9$, while for $\vartheta > 1/2$ we obtain the lower estimates
\beq
\frac{\mathfrak S(\mathcal H)N}{(\log N)^{C_2(\vartheta_0)}} \ \ \text{ if }\ \vartheta > \vartheta_0,
\label{eq:5.19}
\eeq
where the values for $C_2(\vartheta_0)$ are given in the formulation of Theorem~\ref{th:6}.
We remark that we can omit the dependence on~$k$, $u$ and $\mathcal H_k$ since $k \leq 9$, $u$ is bounded and for any admissible $\mathcal H_k$ we have
\beq
\mathfrak S(\mathcal H_k) \geq \prod_{p \leq 2k} \frac1p \left(1 - \frac1p\right)^{-k} \prod_{p > 2k} \left(1 - \frac{k}{p}\right) \left( 1 - \frac1p\right)^{-k} \gg_k 1.
\label{eq:5.20}
\eeq
On the other hand Selberg's sieve (cf.\ Theorem~5.1 of \cite{HR} or Theorem~2 in \S 2.2.2 of \cite{Gre}) yields for the number of $n$'s in $[N, 2N)$ with $P^-\bigl(P_{\mathcal H_k}(n)\bigr) > N^c$ the upper bound
\beq
C(k) \frac{\mathfrak S(\mathcal H_k)N}{(\log (N^c))^k} = C'(k) \frac{\mathfrak S(\mathcal H_k)N}{(\log N)^k},
\label{eq:5.19uj}
\eeq
since $c$ can be chosen as a small constant depending on~$k$.
This shows that the lower estimate \eqref{eq:5.17} is sharp up to a constant factor and also that a positive proportion of all almost prime $k$ tuples, that is all $n \in [N, 2N)$ with
\beq
P^-(n + h_i) > N^c \ \text{ for } \ 1 \leq i \leq k
\label{eq:5.20uj}
\eeq
satisfy that we can find among the elements $n + h_i$ at least one prime and another one with an odd number of prime factors $(\lambda(n + h_j) = -1)$.

We can summarize the unconditional case (cf.\ \eqref{eq:5.4} and \eqref{eq:5.17}) as

\begin{theorem}
\label{th:9}
We have an infinite set $\mathcal D_1$ of positive even numbers $d$ with lower density $\geq 1/315$, including at least one $d \leq 30$, absolute constants $C$, $c$, $c'$ and an \emph{odd} integer $b$, such that there are at least
\beq
c' \frac{N}{\log^9 N}
\label{eq:5.21}
\eeq
primes up to $N$ $(N > N_0(d))$ with $\Omega(p + d) = b$, $P^-(p + d) > p^c$.
\end{theorem}

The fact that for a positive proportion of the almost prime $k$-tuples with \eqref{eq:5.20uj} we have at least one prime among $n  +  h_i$ and another one with \hbox{$\lambda(n + h_j) = -1$}, make possible a proof of the following extension of Theorem~\ref{th:9}.

\begin{theorem}
\label{th:10}
We have absolute constants $C$, $c$, an \emph{odd} $b \leq C$ and a set $\mathcal D_1$ of even numbers with lower density $\geq 1/315$, including at least one $d \in \mathcal D_1$ with $d \leq 30$ such that the set $\mathcal P(d)$ of primes $p$ with $\Omega(p + d) = b$, $P^-(p + d) > p^c$ contains arbitrarily long arithmetic progressions.
\end{theorem}

We will omit the proof of Theorem~\ref{th:10}, since it follows the same line of arguments as \cite{Pin}, where it was proved that if the primes have a distribution level $\vartheta > 1/2$ then we have a positive $d \leq C(\vartheta)$ such that there exist arbitrarily long arithmetic progressions of primes $p$ such that $p + d$ is also a prime (in fact the one, following $p$) for all elements of the progression.

\bigskip

\noindent
{\small J\'anos {\sc Pintz}\\
R\'enyi Mathematical Institute of the Hungarian Academy
of Sciences\\
Budapest\\
Re\'altanoda u. 13--15\\
H-1053 Hungary\\
E-mail: pintz@renyi.hu}

\end{document}